\documentclass{amsart}
\usepackage{amscd,amsmath,amssymb,amsthm,bm,cases,color,comment,mathrsfs}
\usepackage[dvips]{graphicx}
\usepackage[all]{xy}

\theoremstyle{definition}
\newtheorem{definition}{Definition}[section]

\theoremstyle{plain}
\newtheorem{theorem}[definition]{Theorem}
\newtheorem{proposition}[definition]{Proposition}
\newtheorem{lemma}[definition]{Lemma}
\newtheorem{corollary}[definition]{Corollary}

\numberwithin{equation}{section}

\title[Abelian subgroups of the mapping class groups]{Abelian subgroups of the mapping class groups for non-orientable surfaces}

\author[E.~Kuno]{Erika Kuno}
\address{
(Erika Kuno)
Department of Mathematics,
Tokyo Institute of Technology,
2-12-1 Oh-okayama, Meguro-ku, Tokyo 152-8551, Japan
}
\email{kuno.e.aa@m.titech.ac.jp}
\date{\today}

\begin{document}

\begin{abstract}
Birman-Lubotzky-McCarthy proved that any abelian subgroup of the mapping class groups for orientable surfaces is finitely generated.
We apply Birman-Lubotzky-McCarthy's arguments to the mapping class groups for non-orientable surfaces.
We especially find a finitely generated group isomorphic to a given torsion-free subgroup of the mapping class groups.
\end{abstract}

\maketitle

\section{Introduction}\label{Introduction}

Let $S$ be a compact orientable surface of genus $g$ with $b$ boundary components and $c$ connected components.
Assume each connected component of $S$ has a negative Euler characteristic.
Let $S_{g, n}$ be a compact connected orientable surface of genus $g$ with $n$ boundary components.
We also write $S$ as $S_{g, n}$.
We denote by $\mathscr{M}(S)$ the mapping class group of $S$, that is, the group of isotopy classes of orientation preserving self-homeomorphisms of $S$ with isotopies fixing each boundary component of $S$ setwise.
Birman-Lubotzky-McCarthy~\cite{Birman-Lubotzky-McCarthy83} proved that any abelian subgroup of $\mathscr{M}(S)$ is finitely generated with torsion-free rank bounded by $3g+b-3c$.
Kim-Koberda~\cite{Kim-Koberda16} quoted it in their paper and renew it as follows: if $S$ is a compact connected orientable surface and $G$ is a torsion-free abelian subgroup of $\mathscr{M}(S)$, then $G$ is isomorphic to a finitely generated subgroup of $\mathscr{M}(S)$ which consists of Dehn twists and pseudo-Anosov elements on some connected subsurface on $S$ whose supports are pairwise disjoint.
Let $N=N_{g, n}$ be a compact connected non-orientable surfaces of genus $g\geq 1$ with $n\geq 0$ boundary components whose Euler characteristic is negative, that is, $g+n\geq 3$.
Similarly, we also denote by $\mathscr{M}(N)$ the mapping class group of $N$.
Let $p\colon S_{g-1, 2n}\rightarrow N_{g, n}$ be the double covering map of $N_{g, n}$.
Because $\mathscr{M}(N_{g, n})$ is a subgroup of $\mathscr{M}(S_{g-1, 2n})$, it follows that any abelian subgroup of $\mathscr{M}(N_{g, n})$ is finitely generated with torsion-free rank bounded by $3(g-1)+2b$ by the result of Birman-Lubotzky-McCarthy.
In this paper, we write $\iota\colon \mathscr{M}(N_{g, n})\rightarrow \mathscr{M}(S_{g-1, 2n})$ as the injective homomorphism.
By Szepietowski~\cite{Szepietowski10} the image $\iota(\mathscr{M}(N_{g, n}))$ includes no Dehn twists in $\mathscr{M}(S_{g-1, 2n})$.
Therefore, our motivation in this paper is to detect finitely generated groups isomorphic to torsion-free abelian subgroups of $\mathscr{M}(N_{g, n})$.
Applying Birman-Lubotzky-McCarthy's arguments to the mapping class groups of non-orientable surfaces, we obtain the following theorem:

\begin{theorem}\label{first_thm}
Let $N$ be a non-orientable surface whose Euler characteristic is negative and $G$ a torsion-free abelian subgroup of $\mathscr{M}(N)$.
Then, $G$ is isomorphic to a subgroup $\langle\tau_{1},\cdots,\tau_{k}\rangle<\mathscr{M}(N)$, where each $\tau_{i}$ is an isotopy class of a Dehn twist and the supports of $\tau_{i}$ and $\tau_{j}$ are disjoint for $i\not=j$.
Further, $k\leq \frac{3}{2}(g-1)+n-2$ if $g$ is odd and $k\leq \frac{3}{2}g+n-3$ if $g$ is even (see Figure~\ref{fig_two_sided_scc}).
\end{theorem}

\begin{figure}[h]
\includegraphics[scale=0.50]{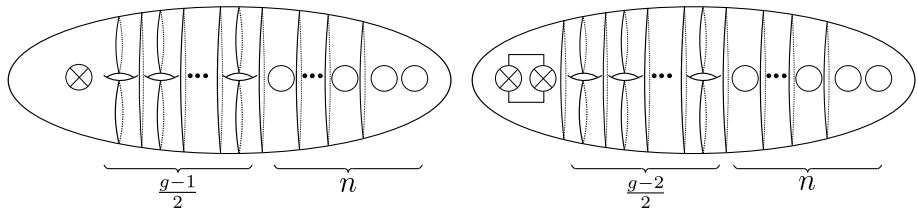}
\caption{The maximal number of two-sided curves on odd genus surfaces (left-hand side) and even genus surfaces (right-hand side).}\label{fig_two_sided_scc}
\end{figure}

Atalan-Szepietowski~\cite[Remark 2.4]{Atalan-Szepietowski14} proved that for odd genus non-orientable surfaces of genus $g\geq 5$ with $n$ punctures, the maximal rank of abelian subgroups of the mapping class groups is $\frac{3}{2}(g-1)+n-2$.
Thus, we give another proof of this result for odd genus non-orientable surfaces whose Euler characteristic are negative.
However, for even genus non-orientable surfaces, we don't know the maximal rank of them (Atalan~\cite[Proposition 3.1]{Atalan15} gave a partial answer for it).
We give the answer for this question.

Thurston~\cite{Thurston88} proved that every mapping class $\tau\in\mathscr{M}(N)$ is either reducible or of finite order or pseudo-Anosov, and if $\tau$ is reducible, then it has a family $\mathscr{A}$ of isotopy classes of essential simple closed curves such that $\tau(\mathscr{A})=\mathscr{A}$ and each of the restrictions of $\tau$ is of finite order or pseudo-Anosov on each connected component of $N-A$, where $A$ is a set of representatives of $\mathscr{A}$.
We call this theorem Thurston's theorem.
In the theorem, the system $A$ is not unique in general.
In Section~\ref{Essential reduction classes}, we will introduce an ``essential reduction system" on $N$ in a similar way to Birman-Lubotzky-McCarthy~\cite{Birman-Lubotzky-McCarthy83}.
Birman-Lubotzky-McCarthy proved that the essential reduction system satisfies the condition in Thurston's theorem and it is a minimal reduction system among such systems and unique up to isotopy for only orientable surfaces.
We show the same result for non-orientable surfaces:

\begin{theorem}\label{second_thm}
A system $A$ satisfying the conditions of Thurston's theorem, which is minimal among such systems, is unique up to isotopy.
\end{theorem}

Theorem~\ref{second_thm} was first proven by Wu~\cite{Wu87}.
We give another proof by applying the arguments of Birman-Lubotzky-McCarthy.

Combining \cite[Theorem A]{Birman-Lubotzky-McCarthy83} and Theorem~\ref{first_thm}, we obtain the following result.

\begin{corollary}
Let $G$ be an abelian subgroup of $\mathscr{M}(N)$.
Then $G$ is finitely generated with torsion-free rank bounded by $\frac{3}{2}(g-1)+n-2$ if $g$ is odd and $\frac{3}{2}g+n-3$ if $g$ is even.
\end{corollary}

There are several differences from the case of orientable surfaces in the proofs of the lemmas to prove Theorems~\ref{first_thm} and \ref{second_thm}.
First difference appears in the proof of (1) in Lemma~\ref{kernel_of_reduction_homomorphism}.
We use the result of Stukow~\cite{Stukow09} about the kernel of the inclusion homomorphism from the mapping class group of a subsurface to that of the ambient surface for the proof.
Secondly, the statement of Lemma~\ref{existance_of_curves_for_fixed_arc_by_homeo} is different from the orientable surface case (\cite[Lemma 2.4]{Birman-Lubotzky-McCarthy83}).
Thirdly, because we use Lemma~\ref{existance_of_curves_for_fixed_arc_by_homeo} to prove Lemma~\ref{the_minimal_adequate_reduction_system}, we have some differences from the orientable surface case in the proof of Lemma~\ref{the_minimal_adequate_reduction_system}.
However, we obtain a similar result to the corresponding lemma (\cite[Lemma 2.5]{Birman-Lubotzky-McCarthy83}).

The paper is organized as follows.
In Section~\ref{Essential reduction classes} we will introduce essential reduction systems in accordance with orientable surface case by Birman-Lubotzky-McCarthy, and prove Theorem~\ref{second_thm}.
In Section~\ref{Abelian subgroups of mapping class groups} we will accomplish the proof of Theorem~\ref{first_thm}.

\section{Essential reduction classes}\label{Essential reduction classes}

In this section, following Birman-Lubotzky-McCarthy we introduce essential reduction systems for non-orientable surfaces.
A compact connected {\it non-orientable surface} of genus $g\geq 1$ with $n\geq 0$ boundary components is the connected sum of $g$ projective planes with $n$ open disks removed.
We denote it by $N=N_{g,n}$.
Note that $N$ is homeomorphic to the surface obtained from a sphere by removing $g+n$ open disks and attaching $g$ M\"{o}bius bands along their boundaries, and we call each of the M\"{o}bius bands the crosscap.
We identify antipodal points of each periphery of a crosscap.
A simple closed curve on $N$ is essential if it does not bound a disk or a M\"{o}bius band, and is not parallel to a boundary component of $N$.
We often refer to essential simple closed curves as curves.
The collection of non-oriented isotopy classes of curves in $N$ is denoted by the symbol $\mathscr{S}(N)$.
The mapping class group $\mathscr{M}(N)$ of $N$ is the group consists of the isotopy classes of self-homeomorphisms on $N$.
We remark that admissible isotopies fix each boundary component setwise.
If $\tau\in\mathscr{M}(N)$ and $\alpha\in\mathscr{S}(N)$, then $\tau(\alpha)$ denotes the class of $t(a)$, where $t\in\tau$ and $a\in\alpha$.
A subset $\mathscr{A}\subset\mathscr{S}(N)$ is {\it admissible} if a set $A$ of the representatives of $\mathscr{A}$ can be chosen so that it consists of pairwise disjoint curves.
Similarly we say that $A$ is an {\it admissible set of representatives}.
Let $\mathscr{A}$ be an admissible subset of $\mathscr{S}(N)$.
From now we use some notations which are the same as Birman-Lubotzky-McCarthy~\cite{Birman-Lubotzky-McCarthy83}. 
The symbol $\mathscr{M}_{\mathscr{A}}(N)$ denotes the stabilizer of $\mathscr{A}$ in $\mathscr{M}(N)$ which preserves the set $\mathscr{A}$.
We denote by $N_{\mathscr{A}}$ the natural compactification of $N-A$, where $A$ is any admissible set of representatives of $\mathscr{A}$.
If $\tau\in\mathscr{M}_{\mathscr{A}}(N)$, then we can choose an admissible set $A$ of representatives of $\mathscr{A}$ and a representative $t$ of $\tau$ such that $t(A)=A$.
Furthermore, $t|_{N-A}$ extends uniquely to $N_{\mathscr{A}}$.
Note that this process determines a well defined class $\widehat{\tau}\in\mathscr{M}(N_{\mathscr{A}})$.
We shall refer to this class $\widehat{\tau}$ as the {\it reduction} of $\tau$ along $\mathscr{A}$.
The assignment $\tau\rightarrow\widehat{\tau}$ yields a homomorphism $\Lambda\colon\mathscr{M}_{\mathscr{A}}(N)\rightarrow\mathscr{M}(N_{\mathscr{A}})$, which we shall refer to as the {\it reduction homomorphism}.
Let $a$ be a two-sided simple closed curve on $N$.
We denote by $t_{a}$ the {\it Dehn twist} along $a$, which is a homeomorphism on $N$ defined by cutting $N$ along $a$, twisting one side by $2\pi$, and reglueing.
Let $\tau_{\alpha}$ be the isotopy class of $t_{a}$, where $\alpha$ is an isotopy class of $a$.
Abusing the notation we often call $\tau_{\alpha}$ the Dehn twist along $a$

We remark that $\Lambda$ is not an isomorphism in general, because according to our definition of the mapping class group each Dehn twist $\tau_{\alpha}$ can be in the kernel ${\rm Ker}(\Lambda)$ of $\Lambda$.
A natural representation $\partial\colon\mathscr{M}(N)\rightarrow{\rm Aut}(\partial N)$ arises from the permutation of boundary components.
Let $\mathscr{A}^{\rm two}$ and $\mathscr{A}^{\rm one}$ be the subsets of $\mathscr{A}$ which consist of all isotopy classes of two-sided curves and one-sided curves respectively.
If we write card($\mathscr{A}$), then it means the cardinality of $\mathscr{A}$.

\begin{lemma}\label{kernel_of_reduction_homomorphism}
Let $\mathscr{A}$ be an admissible subset of $\mathscr{S}(N)$.
Then the following occur.
\begin{itemize}
\item[(1)] ${\rm Ker}(\Lambda)=\langle\tau_{\alpha}\mid\tau_{\alpha}\in\mathscr{A}^{\rm two}\rangle$.
\item[(2)] ${\rm Ker}(\Lambda)\subset {\rm center}({\rm Ker}(\partial\circ\Lambda))$.
\end{itemize}
\end{lemma}

\begin{proof}[Proof of Lemma~\ref{kernel_of_reduction_homomorphism}]
We will prove only (1), since the proof of (2) is the same as that of \cite[Lemma 2.1 (2)]{Birman-Lubotzky-McCarthy83}.
Let $A$ be a set of representatives of $\mathscr{A}$ which are pairwise disjoint, and $\widehat{N-A}$ the surface which we cap the components of $\partial A$ in $N-A$ by once punctured disks.
We put orientations in the new boundary components of $N_{\mathscr{A}}$.
We denote by $\{p_{i}^{+}, p_{i}^{-}\}$ the pair of points on two new boundary components obtained by cutting $N$ along a two-sided curve and compactifying it naturally, and by $q_{i}$ a point on a new boundary component obtained by cutting $N$ along a one-sided curve and compactifying it naturally.
We define ${\rm Homeo}'(N_{\mathscr{A}})$ as the set of mapping classes $\tau$ which satisfy $\tau(\{p_{i}^{+}, p_{i}^{-}\})=\{p_{j}^{+}, p_{j}^{-}\}$ and $\tau(q_{i})=q_{j}$, and if the orientation of the boundary component which has $p_{i}^{+}$ is the same as (resp. opposite to) the boundary component which has $\tau(p_{i}^{+})$, then the orientation of the boundary component which has $p_{i}^{-}$ is the same as (resp. opposite to) the boundary component which has $\tau(p_{i}^{-})$.
Set $\mathscr{M}'(N_{\mathscr{A}})=\pi_{0}({\rm Homeo}_{+}(N_{\mathscr{A}}))$.
We consider the capping homomorphism $\eta_{1}\colon\mathscr{M}'(N_{\mathscr{A}})\rightarrow\mathscr{M}(\widehat{N-A})$ induced by the inclusion $N_{\mathscr{A}}\hookrightarrow\widehat{N-A}$.
Note that the Dehn twists along the new boundary components are not isotopic to the identity in $\mathscr{M}'(N_{\mathscr{A}})$.
We define a homomorphism $\nu_{1}\colon\mathscr{M}(\widehat{N-A})\rightarrow\mathscr{M}(N_{\mathscr{A}})$ by $\nu_{1}(\tau)=\tau|_{N_{\mathscr{A}}}$ for any $\tau\in\mathscr{M}(\widehat{N-A})$.
Since each $\tau\in\mathscr{M}'(N_{\mathscr{A}})$ is compatible with regluing $N_{\mathscr{A}}$, we have a homomorphism $\nu_{2}$ from $\mathscr{M}'(N_{\mathscr{A}})$ to $\mathscr{M}_{\mathscr{A}}(N)$.
This process determines the homomorphism $\nu_{2}$ from $\mathscr{M}'(N_{\mathscr{A}})$ to $\mathscr{M}_{\mathscr{A}}(N)$.
Let $\nu_{3}$ be a restriction of $\nu_{2}$ to ${\rm Ker}(\eta_{1})$.
Then, we obtain the following commutative diagram.
\[\xymatrix{
1 \ar[r] & {\rm Ker}(\eta_{1}) \ar[r] \ar@{->>}[d]^{\nu_{3}} \ar@{}[dr]|\circlearrowleft & \mathscr{M}'(N_{\mathscr{A}}) \ar[r]^{\eta_{1}} \ar@{->>}[d]^{\nu_{2}} \ar@{}[dr]|\circlearrowleft & \mathscr{M}(\widehat{N-A}) \ar[d]^{\nu_{1}}_{\cong} \\
1 \ar[r] & {\rm Ker}(\Lambda) \ar[r] & \mathscr{M}_{\mathscr{A}}(N) \ar[r]^{\Lambda} & \mathscr{M}(N_{\mathscr{A}}) \\
}\]
The homomorphism $\nu_{1}\colon\mathscr{M}(\widehat{N-A})\rightarrow\mathscr{M}(N_{\mathscr{A}})$ is isomorphism by our definition of the mapping class groups.
The homomorphism $\nu_{2}\colon\mathscr{M}'(N_{\mathscr{A}})\rightarrow\mathscr{M}_{\mathscr{A}}(N)$ is surjective.
Actually, for any $\tau\in\mathscr{M}_{\mathscr{A}}(N)$ there exists a representative $t$ of $\tau$ and an admissible subset $A$ of representatives of $\mathscr{A}$ such that $t(A)=A$ and $t$ maps each two-sided (resp. one-sided) curve to a two-sided (resp. one-sided) curve.
Hence $\tau|_{N_{\mathscr{A}}}$ is an element of $\mathscr{M}'(N_{\mathscr{A}})$, and so it is a lift of $\tau$ to $\mathscr{M}'(N_{\mathscr{A}})$.
The homomorphism $\nu_{3}\colon{\rm Ker}(\eta_{1})\rightarrow{\rm Ker}(\Lambda)$ is surjective by the five-lemma.
Thus, ${\rm Ker}(\Lambda)$ is generated by at most the projections of the generating set of ${\rm Ker}(\eta_{1})$.
By Theorem~\ref{kernel_of_reduction_homomorphism}, we have ${\rm Ker}(\eta_{1})=\langle\tau_{\alpha}\mid\alpha\in\mathscr{A}\rangle\cong\mathbb{Z}^{N}$, where $N=2{\rm card}(\mathscr{A}^{\rm two})+{\rm card}(\mathscr{A}^{\rm one})$.
We know that for each regular neighborhood $\alpha$ of a one-sided curve on $N$, the Dehn twist $\tau_{\alpha}$ along $\alpha$ is not contained in ${\rm Ker}(\Lambda)$ since it is a trivial element in $\mathscr{M}(N_{\mathscr{A}})$.
Moreover, both two Dehn twists along the boundary components in $N_{\mathscr{A}}$ which comes from the same $\alpha\in\mathscr{A}$ as the boundary components of the regular neighborhood of $\alpha$ in $N$ project $\tau_{\alpha}$ in ${\rm Ker}(\Lambda)$.
Therefore ${\rm Ker}(\Lambda)=\langle\tau_{\alpha}\mid\alpha\in\mathscr{A}^{\rm two}\rangle$.
\end{proof}

Let $\Gamma(N_{\mathscr{A}})=\{N_{i}\mid i\in I\}$ be the set of the connected components of $N_{\mathscr{A}}$, and so $N_{\mathscr{A}}=\coprod_{i\in I} N_{i}$.
There is a natural representation $\varphi\colon\mathscr{M}(N_{\mathscr{A}})\rightarrow{\rm Aut}(\partial N_{\mathscr{A}})$ which arises from the permutation of components.
The kernel ${\rm Ker}(\varphi)$ of $\varphi$ is isomorpic to $\bigoplus_{i\in I}\mathscr{M}(N_{i})$.
If $\tau\in\mathscr{M}(N_{\mathscr{A}})$, then for some exponent $n$, $\tau^{n}$ is contained in ${\rm Ker}(\varphi)$.
For any such an exponent, we refer to the element of $\mathscr{M}(N_{i})$ obtained by restricting $\tau^{n}$ as {\it restrictions} of $\tau$.

Let $S$ be the double covering orientable surface of $N$.
Wu~\cite{Wu87} proved that a mapping class $\tau\in\mathscr{M}(N)$ is of {\it finite order} (resp. {\it reducible}, {\it pseudo-Anosov}) if $\iota(\tau)\in\mathscr{M}(S)$ is of {\it finite order} (resp. {\it reducible}, {\it pseudo-Anosov}).
Moreover according to Thurston~\cite{Thurston88}, there is Nielsen-Thurston classification for $\mathscr{M}(N)$ (the proof is found in the paper of Wu~\cite{Wu87}):

\begin{theorem}{\rm(}\cite{Thurston88}, \cite[Theorem 2]{Wu87}{\rm)}
A homeomorphism $t$ on $N$ is reduced and is not of finite order if and only if $t$ is isotopic to a pseudo-Anosov homeomorphism.
\end{theorem}

Then we define pseudo-Anosov and reducible mapping classes on non-orientable surfaces according to Birman-Lubotzky-McCarthy~\cite{Birman-Lubotzky-McCarthy83} as follows:

\begin{definition}\label{definition_of_pA_and_reducible}
A mapping class $\tau\in\mathscr{M}(N_{\mathscr{A}})$ is {\it pseudo-Anosov} if $\mathscr{S}(N_{i})\not=\emptyset$ for every $i\in I$ and $\tau^{n}(\alpha)\not=\alpha$ for any $\alpha\in\mathscr{S}(N_{\mathscr{A}})$ and any $n\not=0$.
A mapping class $\tau\in\mathscr{M}(N_{\mathscr{A}})$ is {\it reducible} if there exists an admissible subset $\mathscr{A}$ such that $\tau(\mathscr{A})=\mathscr{A}$.
\end{definition}

We call the admissible set $\mathscr{A}$ as in Definition~\ref{definition_of_pA_and_reducible} a {\it reduction system} for $\tau$, and we say each $\alpha\in\mathscr{A}$ as a {\it reduction class} for $\tau$.
A reduction system $\mathscr{A}$ for $\tau$ is an {\it adequate reduction system} if the restrictions of $\tau$ to each component of $N_{\mathscr{A}}$ are either of finite order or pseudo-Anosov.
If $\tau\in\mathscr{M}(N)$ is either of finite order or pseudo-Anosov mapping class on $N$, we call $\tau$ is {\it adequately reduced}. 

We restate Thurston's theorem by using adequate reduction system as follows:

\begin{theorem}{\rm(}\cite{Thurston88}{\rm)}\label{thurston's_theorem}
Every mapping class $\tau\in\mathscr{M}(N)$ is either reducible or adequately reduced.
If $\tau$ is reducible, then there exists an adequate reduction system $\mathscr{A}$ for $\tau$.
\end{theorem}

Let $i\colon\mathscr{S}(N)\times\mathscr{S}(N)\rightarrow\mathbb{N}\cup\{0\}$ be the geometric intersection number.
A reduction class $\alpha$ for $\tau$ is {\it essential} if for each $\beta\in\mathscr{S}(N)$ such that $i(\alpha, \beta)\not=0$ and each integer $m\not=0$, the class $\tau^{m}(\beta)$ is distinct from $\beta$.
We often say $\alpha$ is essential if $\alpha$ is an essential reduction system for some $\tau$.

\begin{proposition}{\rm(}\cite[Proposition 2.3]{Birman-Lubotzky-McCarthy83}{\rm)}
Let $\alpha\in\mathscr{A}$ and $\alpha'\in\mathscr{A}'$ be reduction classes for $\tau\in\mathscr{M}(N)$.
Suppose that $\alpha$ is essential.
Then $i(\alpha, \alpha')=0$.
\end{proposition}

The following lemma is different from the corresponding lemma (\cite[Lemma 2.4]{Birman-Lubotzky-McCarthy83}) by Birman-Lubotzky-McCarthy.

\begin{lemma}\label{existance_of_curves_for_fixed_arc_by_homeo}
Let $F$ be a compact connected orientable or non-orientable surface with $\chi(F)<0$.
Fix any isotopy class $\delta$ of properly embedded arc $d$ on $F$, namely, $\partial d$ is embedded in $\partial F$ and the interior of $d$ is embedded in the interior of $F$.
We choose any $\tau\in\mathscr{M}(F)$ with $\tau(\delta)=\delta$.
Then one of the following occurs.
\begin{itemize}
\item[(1)] $F$ is either $S_{0, 3}$ or $N_{1, 2}$ or $N_{2, 1}$.
\item[(2)] If $\delta$ is an isotopy class of an arc which connects distinct two boundary components, then there exists $\gamma\in\mathscr{S}(F)$ such that $\tau(\gamma)=\gamma$ and $i(\alpha, \gamma)\not=0$ for any $\alpha\in\mathscr{S}(F)$ with $i(\alpha, \delta)\not=0$.
\item[(3)] If $\delta$ is an isotopy class of an arc which connects one boundary component, goes through crosscaps even number of times, and surrounds one crosscap, then for any $\alpha\in\mathscr{S}(F)$ excepting $\beta_{0}$ which is shown in Figure~\ref{fig_beta0_beta1_beta2} with $i(\alpha, \delta)\not=0$ there exists $\gamma\in\mathscr{S}(F)$ such that $\tau(\gamma)=\gamma$ and $i(\alpha, \gamma)\not=0$.
\item[(4)] If $\delta$ is an isotopy class of an arc which connects one boundary component, goes through crosscaps even number of times, and does not surround one crosscap, then for any $\alpha\in\mathscr{S}(F)$ with $i(\alpha, \delta)\not=0$ there exists $\gamma\in\mathscr{S}(F)$ such that $\tau(\gamma)=\gamma$ and $i(\alpha, \gamma)\not=0$.
\item[(5)] If $\delta$ is an isotopy class of an arc which connects one boundary component and goes through crosscaps odd number of times, then for any $\alpha\in\mathscr{S}(F)$ excepting $\beta_{1}$ and $\beta_{2}$ which are shown in Figure~\ref{fig_beta0_beta1_beta2} with $i(\alpha, \delta)\not=0$ there exists $\gamma\in\mathscr{S}(F)$ such that $\tau(\gamma)=\gamma$ and $i(\alpha, \gamma)\not=0$.
\end{itemize}
\end{lemma}

\begin{proof}
Let $d$ be a properly embedded representative of $\delta$, and $\eta(d)$ a regular neighborhood of $d$ with the boundary components of $F$ which have the end points of $d$.
If $d$ connects distinct two boundary components (Case (a)) or connects one boundary component and goes through crosscaps even number of times (Case (b)), then $\eta(d)$ is homeomorpic to $S_{0, 3}$.
If $d$ connects one boundary component and goes through crosscaps odd number of times (Case (c)), then $\eta(d)$ is homeomorphic to $N_{1, 2}$.
In Case (a), two of the boundary component of $\eta(d)$ are those of $F$.
We denote by $\gamma$ the isotopy class of the other boundary component of $\eta(d)$.
Then $\gamma$ is not isotopic to a point since $\chi(F)<0$.
If $\gamma$ is isotopic to a component of $\partial F$ (resp. a crosscap), then $F$ is homeomorphic to $S_{0, 3}$ (resp. $N_{1, 2}$).
We suppose that $\gamma$ is an essential curve.
Since $\tau(\delta)=\delta$, it follows that $\tau(\gamma)=\gamma$.
If $\alpha\in\mathscr{S}(F)$ intersects $\delta$ nontrivially, then $i(\alpha, \gamma)\not=0$.
In Case (b), one of the boundary component of $\eta(d)$ is that of $F$.
We put $\gamma_{2}$ and $\gamma_{3}$ as the isotopy classes of the other two components.
Neither $\gamma_{2}$ nor $\gamma_{3}$ is isotopic to a point.
Firstly we suppose that $\gamma_{2}$ is parallel to $\partial F$.
If $\gamma_{3}$ is isotopic to a component of $\partial F$ (resp. a crosscap), then $F$ is homeomorphic to $S_{0, 3}$ (resp. $N_{1, 2}$).
We suppose that $\gamma_{3}$ is an essential curve.
Then it follows that $\tau(\gamma_{3})=\gamma_{3}$ and $i(\alpha, \gamma_{3})\not=0$ for any $\alpha\in\mathscr{S}(F)$ with $i(\alpha, \delta)\not=0$.
Next we suppose that $\gamma_{2}$ is isotopic to a crosscap.
If $\gamma_{3}$ is isotopic to a component of $\partial F$ (resp. a crosscap), then $F$ is homeomorphic to $N_{1, 2}$ (resp. $N_{2, 1}$).
We suppose that $\gamma_{3}$ is an essential curve.
Unless $\alpha$ is isotopic $\beta_{0}$ in Figure~\ref{fig_beta0_beta1_beta2}, we see $i(\alpha, \gamma_{3})\not=0$ for any $\alpha\in\mathscr{S}(F)$ with $i(\alpha, \delta)\not=0$.
We can show similar results if $\gamma_{2}$ is essential.
In Case (c), one of the boundary component of $\eta(d)$ is that of $F$.
We denote by $\gamma$ the isotopy class of the other boundary component of $\eta(d)$.
We see $\gamma$ is not isotopic to a point.
If $\gamma$ is isotopic to a component of $\partial F$ (resp. a crosscap), then $F$ is homeomorphic to $N_{1 ,2}$ (resp. $N_{2, 1}$).
We suppose that $\gamma$ is an essential curve.
Since $\tau(\delta)=\delta$, it follows that $\tau(\gamma)=\gamma$.
Note that $\gamma$ bounds $N_{1, 1}$.
We can take only $\beta_{1}$ and $\beta_{2}$ which intersect $\delta$ and do not intersect $\gamma$ as shown in Figure~\ref{fig_beta0_beta1_beta2}.
Unless $\alpha$ is isotopic to $\beta_{1}$ or $\beta_{2}$ in Figure~\ref{fig_beta0_beta1_beta2}, we see $i(\alpha, \gamma)\not=0$ for any $\alpha\in\mathscr{S}(F)$ with $i(\alpha, \delta)\not=0$.
\end{proof}

\begin{figure}[h]
\includegraphics[scale=0.60]{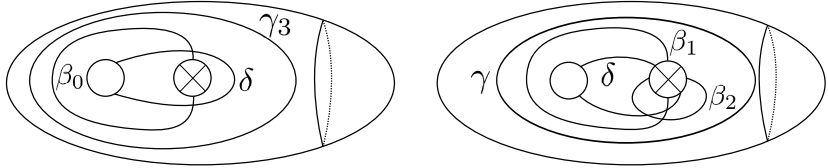}
\caption{Curves $\beta_{0}$, $\beta_{1}$, and $\beta_{2}$.}\label{fig_beta0_beta1_beta2}
\end{figure}

The result of the following lemma is the same as the orientable surface case by Birman-Lubotzky-McCarthy (\cite[Lemma 2.5]{Birman-Lubotzky-McCarthy83}), while the statement of Lemma~\ref{existance_of_curves_for_fixed_arc_by_homeo} is different from the orientable surface case by them (\cite[Lemma 2.4]{Birman-Lubotzky-McCarthy83}).

\begin{lemma}\label{the_minimal_adequate_reduction_system}
Let $\mathscr{A}$ be an adequate reduction system for $\tau$ and let $\alpha\in\mathscr{A}$.
Set $\mathscr{A}'=\mathscr{A}-\{\alpha\}$.
Then the following are equivalent:
\begin{itemize}
\item[(1)] $\alpha$ is essential.
\item[(2)] $\mathscr{A}'$ is not an adequate reduction system for $\tau^{m}$ for any $m\not=0$.
\end{itemize}
\end{lemma}

\begin{proof}
We omit the proof that (1) implies (2) because it is the same as that of Birman-Lubotzky-McCarthy.
We assume that $\alpha\in\mathscr{A}$ is not an essential reduction class for $\tau$.
Then there exists a class $\gamma\in\mathscr{S}(N)$ with $i(\alpha, \gamma)\not=0$ and $n\not=0$ such that $\tau^{n}(\gamma)=\gamma$.
Let $A$ be an admissible set of the representatives of $\mathscr{A}$.
We cut $N$ along $A$, then $\gamma$ determines a finite set of pairwise disjoint isotopy classes of properly embedded arcs in $N_{\mathscr{A}}$, which we denote by $\widehat{\gamma}$.
If a component of $N_{\mathscr{A}}$ has boundary components arising from $\alpha$, then we shall say the component is bounded by $\alpha$.
We note that at least one element of $\widehat{\gamma}$ occurs in a component of $N_{\mathscr{A}}$ bounded by $\alpha$.
We put $\Lambda\colon\mathscr{M}_{\mathscr{A}}(N)\rightarrow\mathscr{M}(N_{\mathscr{A}})$, and set $\widehat{\tau}=\Lambda(\tau)$.
Since $\tau^{n}(\gamma)=\gamma$, therefore $\widehat{\tau}^{n}(\widehat{\gamma})=\widehat{\gamma}$.
By choosing a larger exponent $n$ if necessary, we may assume that $\widehat{\tau}^{n}$ preserves each component of $N_{\mathscr{A}}$, $\partial N_{\mathscr{A}}$, and $\widehat{\gamma}$.
In particular, the restrictions of $\widehat{\tau}^{n}$ to the components of $N_{\mathscr{A}}$ bounded by $\alpha$ preserve a nontrivial isotopy class of a properly embedded arc.
By Lemma~\ref{existance_of_curves_for_fixed_arc_by_homeo}, for each such component, either the corresponding restriction of $\widehat{\tau}^{n}$ is reducible or the component is $S_{0, 3}$, $N_{1, 2}$, or $N_{2, 1}$.
Now we suppose that $\widehat{\tau}^{n}$ is adequately reduced, therefore all the restrictions $\widehat{\tau}^{n}$ have no curves which preserve.
Then it follows that each component bounded by $\alpha$ is either $S_{0, 3}$ or $N_{1, 2}$ or $N_{2, 1}$.
A pair of pants will not support a pseudo-Anosov mapping class.
We can show that the mapping classes on $N_{1, 2}$ (resp. $N_{2, 1}$) which preserve a properly embedded essential arc on it and each component of $\partial N_{1, 2}$ (resp. $\partial N_{2, 1}$) are of finite order.
Thus by choosing a larger exponent $n$ if necessary, we may assume that the restrictions of $\widehat{\tau}^{n}$ to the component bounded $\alpha$ are trivial.

From now we consider the corresponding situation when we reduce along $\mathscr{A}'$.
We define the reduction homomorphism $\Lambda'\colon\mathscr{M}_{\mathscr{A}'}(N)\rightarrow\mathscr{M}(N_{\mathscr{A}'})$.
Let $\widehat{\gamma}'$ and $\widehat{\alpha}'$ be the lift of $\gamma$ and $\alpha$ on $N_{\mathscr{A}'}$ respectively.
We see $\widehat{\alpha}'$ is an adequate reduction class for $\Lambda'(\tau^{n})$.
We denote by $L$ the component of $N_{\mathscr{A}'}$ which includes $\widehat{\alpha}'$.
We set $\nu=\Lambda'(\tau^{n})|_{L}$.
We also define the reduction homomorphism $\Lambda''\colon\mathscr{M}_{\widehat{\alpha}'}(L)\rightarrow\mathscr{M}(L_{\widehat{\alpha}'})$, and set $\nu''=\Lambda''(\nu)$.
Because the restriction of $\widehat{\tau}^{n}$ to the component of $N_{\mathscr{A}}$ bounded by $\alpha$ is identity, $\nu''$ is identity.
If $\widehat{\alpha}'$ is an isotopy class of one-sided curve, then $\nu$ is also identity by Lemma~\ref{kernel_of_reduction_homomorphism}.
If $\widehat{\alpha}'$ is an isotopy class of two-sided curve, then ${\rm Ker}(\Lambda'')=\langle\tau_{\widehat{\alpha}'}\rangle$ by Lemma~\ref{kernel_of_reduction_homomorphism}, and so there exists $k\not=0$ such that $\nu=\tau_{\widehat{\alpha}'}^{k}$.
We will prove that $k=0$ from now.
First we assume that $\gamma$ intersects only $\alpha$, that is, $\gamma\subset N-A'$ ($A'$ is an admissible set of the representatives of $\mathscr{A}'$). 
We have $i(\widehat{\alpha}', \widehat{\gamma}')\not=0$, so $\widehat{\gamma}'\in\mathscr{S}(L)$.
Since $\tau^{n}(\gamma)=\gamma$, it follows that $\nu(\widehat{\gamma}')=\widehat{\gamma}'$.
Then we see $k$ should be $0$.
Next we assume that $\gamma$ intersects other curves of $\mathscr{A}'$, that is, $i(\gamma, \mathscr{A}')\not=0$.
In this case $\widehat{\gamma}'$ is the family consists of isotopy classes of arcs which go through $L$ at least once.
If $L=S_{0, 3}$, then $\nu$ is identity by our assumption.
If $L=N_{1, 2}$, then $k$ has to be $0$ because $\mathscr{M}(N_{1, 2})$ is finite.
If $L=N_{2, 1}$, then there is only one kind of two-sided curve on $L$ shown in Figure~\ref{fig_hat_alpha}.
By Lemma~\ref{existance_of_curves_for_fixed_arc_by_homeo}, $\widehat{\alpha}'$ can not be $\beta_{0}$, $\beta_{1}$, or $\beta_{2}$ because $\widehat{\alpha}'$ is an isotopy class of two-sided curve.
Then there exists $\delta\in\mathscr{S}(L)$ such that $i(\widehat{\alpha}', \delta)\not=0$ and $\nu(\delta)=\delta$.
Hence $k=0$.
Otherwise, similarly to the previous case $\widehat{\alpha}'$ can not be $\beta_{0}$, $\beta_{1}$, or $\beta_{2}$, and so there exists $\delta\in\mathscr{S}(L)$ such that $i(\widehat{\alpha}', \delta)\not=0$ and $\nu(\delta)=\delta$.
Hence we obtain $k=0$.
From the above arguments, we prove $\mathscr{A}'$ is an adequate reduction system for $\tau$.
\end{proof}

We set $\mathscr{A}_{\tau}=\{\alpha\in\mathscr{S}(N)\mid\alpha~{\rm is~an~essential~reduction~class~for}~\tau\}$.

\begin{lemma}{\rm(}\cite[Lemma 2.6]{Birman-Lubotzky-McCarthy83}{\rm)}
\begin{itemize}
\item[(1)] $\sigma(\mathscr{A}_{\tau})=\mathscr{A}_{\sigma\tau\sigma^{-1}}$.
\item[(2)] $\mathscr{A}_{\tau^{m}}=\mathscr{A}_{\tau}$ for all $m\not=0$.
\item[(3)] $\mathscr{A}_{\tau}$ is an adequate reduction system for $\tau$.
\item[(4)] $\mathscr{A}_{\tau}\subset\mathscr{A}$ for each adequate reduction system $\mathscr{A}$ for $\tau$.
\end{itemize}
\end{lemma}

At the end of this section, we prove Theorem~\ref{second_thm}.
The proof is the same as the proof by Birman-Lubotzky-McCarthy for orientable surfaces.

\begin{proof}[Proof of Theorem~\ref{second_thm}]
Let $\tau\in\mathscr{M}(N)$.
Then, by Theorem~\ref{thurston's_theorem} either $\tau$ is adequately reduced (that is the case $\mathscr{A}=\emptyset$) or $\tau$ is reducible, and if $\tau$ is reducible, then there exists an adequate reduction system.
By Lemma~\ref{the_minimal_adequate_reduction_system}, $\mathscr{A}_{\tau}$ is the intersection of all adequate reduction systems for $\tau$.
Hence $\mathscr{A}_{\tau}$ is canonical and unique.
The desired curve system $A$ is any representative of $\mathscr{A}_{\tau}$.
\end{proof}

\begin{figure}[h]
\includegraphics[scale=0.40]{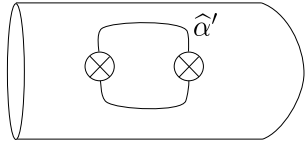}
\caption{The only two-sided curve $\widehat{\alpha}'$ on $N_{2,1}$.}\label{fig_hat_alpha}
\end{figure}

\section{Abelian subgroups of mapping class groups}\label{Abelian subgroups of mapping class groups}

In this section we prove Theorem~\ref{first_thm}.
Let $N$ be a compact connected non-orientable surface, and $\mathscr{A}\in\mathscr{S}(N)$ an admissible subset.
Then $G<\mathscr{M}(N_{\mathscr{A}})$ is {\it adequately reduced} if each $\tau\in G$ is adequately reduced.
Let $G$ be an abelian subgroup of $\mathscr{M}(N_{\mathscr{A}})$.
We denote by ${\rm rank}(G)$ the torsion-free rank of $G$.

\begin{lemma}{\rm(}\cite[Lemma 3.1]{Birman-Lubotzky-McCarthy83}{\rm)}\label{A_G_is_an_adequate_reduction_system}
\begin{itemize}
\item[(1)] Let $\mathscr{A}_{G}$ be the union of the essential reduction systems $\mathscr{A}_{\tau}$ for any $\tau\in G$.
Then $\mathscr{A}_{G}$ is an adequate reduction system for each $\tau\in G$.
\item[(2)] If $G$ is adequately reduced, then ${\rm rank}(G)\leq C_{0}(N_{\mathscr{A}})$, where $C_{0}(N_{\mathscr{A}})$ is the number of components of $N_{\mathscr{A}}$ not homeomorphic to a pair of pants.
\end{itemize}
\end{lemma}

\begin{proof}
We can show (1) by a similar argument to Birman-Lubotzky-McCarthy.
Therefore we only give the proof of (2).
We set $N_{\mathscr{A}}=\amalg_{i=1}^{k}N_{i}$, where each $N_{i}$ is a connected component of $N_{\mathscr{A}}$.
If $N_{i}$ is a non-orientable surface, then let $p_{i}\colon S_{i}\rightarrow N_{i}$ be the double covering map of $N_{i}$, where $S_{i}$ is the double covering orientable surface of $N_{i}$.
If $N_{i}$ is an orientable surface, then let $p_{i}\colon S_{i}\rightarrow N_{i}$ be the identity map on $N_{i}$.
We define $p=\amalg_{i=1}^{k}p_{i}\colon\amalg_{i=1}^{k}S_{k}\rightarrow\amalg_{i=1}^{k}N_{i}$ by $p(x)=p_{i}(x)$ for $x\in S_{i}$, then $p$ is a covering map of $N_{\mathscr{A}}$.
Let $\iota\colon\mathscr{M}(N_{\mathscr{A}})\rightarrow\mathscr{M}(\amalg_{i=1}^{k}S_{i})$ be an injective homomorphism induced by $p$.
We choose an adequately reduced abelian subgroup $G<\mathscr{M}(N_{\mathscr{A}})=\mathscr{M}(\amalg_{i=1}^{k}N_{i})$.
We set $\widetilde{G}=\iota(G)<\mathscr{M}(\amalg_{i=1}^{k}S_{i})$.
Then, $\widetilde{G}$ is an adequately reduced abelian subgroup, since $\tau\in\mathscr{M}(N_{i})$ is finite order (resp. pseudo-Anosov) mapping class if and only if $\tilde{\tau}\in\mathscr{M}(S_{i})$ is finite order (resp. pseudo-Anosov) mapping class.
By \cite[Lemma 3.1]{Birman-Lubotzky-McCarthy83}, we know ${\rm rank}(\widetilde{G})\leq C_{0}(\amalg_{i=1}^{k}S_{i})$, where $ C_{0}(\amalg_{i=1}^{k}S_{i})$ is the number of components of $\amalg_{i=1}^{k}S_{i}$ not homeomorphic to a pair of pants.
If $N_{i}$ is a non-orientable surface, then $S_{i}$ can not be a pair of pants, because the Euler characteristic of $S_{i}$ should be even but that of a pair of pants is $-1$.
If $N_{i}$ is an orientable surface, then $N_{i}$ is a pair of pants if and only if $S_{i}$ is a pair of pants.
Hence $C_{0}(\amalg_{i=1}^{k}S_{i})= C_{0}(N_{\mathscr{A}})$.
By $\widetilde{G}\cong G$, it follows that ${\rm rank}(G)\leq C_{0}(N_{\mathscr{A}})$.
\end{proof}

Finally, we give the proof of Theorem~\ref{first_thm}.

\begin{proof}[Proof of Theorem~\ref{first_thm}]
Let $G$ be a torsion-free abelian subgroup of $\mathscr{M}(N)$.
By (1) of Lemma~\ref{A_G_is_an_adequate_reduction_system}, we see $G\in\mathscr{M}_{\mathscr{A}_{G}}(N)$.
We denote by $\Lambda\colon\mathscr{M}_{\mathscr{A}_{G}}(N)\rightarrow\mathscr{M}(N_{\mathscr{A}_{G}})$ the reduction homomorphism along $\mathscr{A}_{G}$.
We set $H=\Lambda(G)$.
The sequence
\[\xymatrix{
1 \ar[r] & G\cap{\rm Ker}(\Lambda) \ar[r] & G \ar[r] & H \ar[r] & 1
}\]
is exact.
From this exact sequence, we have ${\rm rank}(G)={\rm rank}(G\cap{\rm Ker}(\Lambda))+{\rm rank}(H)$.
By Lemma~\ref{kernel_of_reduction_homomorphism}, ${\rm rank}(G\cap{\rm Ker}(\Lambda))\leq{\rm rank}({\rm Ker}(\Lambda))\leq{\rm card}(\mathscr{A}^{\rm two})$.
Therefore, we see $G\cap{\rm Ker}(\Lambda)$ is isomorphic to the group generated by Dehn twists along pairwise disjoint curves whose cardinality is at most ${\rm card}(\mathscr{A}^{\rm two})$.

From (2) of Lemma~\ref{A_G_is_an_adequate_reduction_system}, we see ${\rm rank}(H)\leq C_{0}(N_{\mathscr{A}_{G}})$.
By the arguments in the proof of \cite[Lemma 3.1 (2)]{Birman-Lubotzky-McCarthy83}, we see $H$ is isomorphic to the group generated by pseudo-Anosov mapping classes on connected subsurfaces, and the number is bounded by $C_{0}(N_{\mathscr{A}_{G}})$.
However, we remark that any subsurface which supports a pseudo-Anosov mapping class must also support a non-trivial Dehn twist.
Hence we can replace the pseudo-Anosov mapping class generators by Dehn twists. 
As shown in Figure~\ref{fig_two_sided_scc} (see also \cite[Proposition 2.3]{Atalan-Korkmaz14}), ${\rm card}(\mathscr{A}^{\rm two})$ is bounded by $\frac{3}{2}(g-1)+n-2$ if $g$ is odd and $\frac{3}{2}g+n-3$ if $g$ is even, so we are done.
\end{proof}

\par
{\bf Acknowledgements:} The author is extremely grateful to Hisaaki Endo for his warm encouragement and helpful advice.
She also wishes to thank B\l a\.{z}ej Szepietowski.
He told her Nielsen-Thurston classification for the mapping class group of non-orientable surface and the proof by Wu, the fact that the mapping class group of non-orientable surface is embedded in that of the double covering orientable surface, and the result by him that any Dehn twists are not contained in the image of the embedding.
Further He read her arguments in this paper and gave her a lot of valuable comments, and she can improve Theorem~\ref{first_thm}. 
She also thanks Genki Omori for the useful discussion with her particularly about Lemma~\ref{kernel_of_reduction_homomorphism}.
This work was supported by JSPS KAKENHI, the grant number 16J00397 of Research Fellowship for Young Scientists.



\begin{thebibliography}{99}
\bibitem{Atalan15}
F. Atalan, {\it An algebraic characterization of a Dehn twist for nonorientable surfaces}, available at arXiv:1501.07183v2 [math.GT].

\bibitem{Atalan-Korkmaz14}
F. Atalan and M. Korkmaz, {\it Automorphisms of curve complexes on nonorientable surfaces}, Groups Geom. Dyn. {\bf 8} (2014), no. 1, 39--68.

\bibitem{Atalan-Szepietowski14}
F. Atalan and B. Szepietowski, {\it Automorphisms of the mapping class group of a nonorientable surface}, available at arXiv:1403.2774v2 [math.GT].

\bibitem{Birman-Lubotzky-McCarthy83}
J. S. Birman, A. Lubotzky, and J. McCarthy, {\it Abelian and solvable subgroups of the mapping class groups}, Duke Math.\ J. {\bf 50} (1983), no. 4, 1107--1120.

\bibitem{Kim-Koberda16}
S. Kim and T. Koberda, {\it Right-angled Artin groups and finite subgraphs of curve graphs}, Osaka\ J.\ Math. {\bf 53}, no. 3, to appear.

\bibitem{Szepietowski10}
B. Szepietowski, {\it Embedding the braid group in mapping class groups}, Publ. Mat. {\bf 54} (2010), no. 2, 359--368.

\bibitem{Stukow09}
M. Stukow, {\it Commensurability of geometric subgroups of mapping class groups}, Geom. Dedicata {\bf 143} (2009), 117--142.

\bibitem{Thurston88}
W. Thurston, {\it On the geometry and dynamics of diffeomorphisms of surfaces}, Bull. Amer. Math. Soc. (N.S.) {\bf 19} (1988), no. 2, 417--431.

\bibitem{Wu87}
Y. Wu, {\it Canonical reducing curves of surface homeomorphism}, Acta Math. Sinica (N.S.) {\bf 3} (1987), no. 4, 305--313.

\end{thebibliography}
\end{document}